\documentclass[11pt]{amsart} \textwidth=14.5cm \oddsidemargin=1cm
\usepackage{amsmath,wasysym}
\usepackage{amsxtra}
\usepackage{amscd}
\usepackage{amsthm}
\usepackage{amsfonts}
\usepackage{amssymb}
\usepackage{eucal}
\usepackage{graphics, color}
\usepackage{hyperref,mathrsfs}
\usepackage[usenames,dvipsnames]{xcolor}
\usepackage{tikz}
\usepackage{tikz-cd}
\usepackage{stmaryrd}
\usepackage[all]{xy}
\usepackage{hyperref}
\usepackage{color}
\usepackage{bbm} 
\allowdisplaybreaks

\hypersetup{colorlinks,linkcolor=blue, urlcolor=blue,
	citecolor=blue}

\newtheorem{thm}{Theorem} [section]

\newtheorem{lem}[thm]{Lemma}

\newtheorem{prop}[thm]{Proposition}

\newtheorem{rem}[thm]{Remark}

\numberwithin{equation}{section}

\newcommand{\fC}{\mathbb{C}}

\newcommand{\ve}{\varepsilon}

\newcommand{\HH}{\mathcal{H}}
\newcommand{\Hom}{\mathrm{Hom}}

\newcommand{\K}{\mathbb{K}}
\newcommand{\Ker}{\mathrm{Ker}}

\newcommand{\Mid}{\underline{\mathbf{m}}}
\newcommand{\mi}{\mathbf{i}}
\newcommand{\mj}{\mathbf{j}}
\newcommand{\mk}{\mathbf{k}}
\newcommand{\ml}{\mathbf{l}}
\newcommand{\Par}{\mathrm{Par}}
\newcommand{\N}{\mathbb{N}}
\newcommand{\Nid}{\underline{\mathbf{n}}}

\newcommand{\qU}{\mathbf{U}}

\newcommand{\V}{\mathbb{V}}

\newcommand{\half}{\frac{1}{2}}

\title[Invariant theory of $\imath$quantum groups of type AIII]
{Invariant theory of $\imath$quantum groups of type AIII}

\author[Li Luo]{Li Luo}
\author[Zheming Xu]{Zheming Xu}
\address{School of Mathematical Sciences,  Key Laboratory of MEA(Ministry of Education) \& Shanghai Key Laboratory of PMMP,  East China Normal University, Shanghai 200241, China}
\email{lluo@math.ecnu.edu.cn (Luo), 52265500008@stu.ecnu.edu.cn (Xu)}

\begin{document}
	
	\begin{abstract}We develop an invariant theory of quasi-split $\imath$quantum groups $\qU_n^\imath$ of type AIII on a tensor space associated to $\imath$Howe dualities. The first and second fundamental theorems for $\qU_n^\imath$-invariants are derived.
	\end{abstract}
	
	\maketitle
	\setcounter{tocdepth}{1}
	
	\section{Introduction}
	\subsection{}
	The notion of quantum symmetric pair, consisting of a Drinfeld-Jimbo quantum group and a certain coideal subalgebra (named $\imath$quantum group), was introduced by Letzter \cite{Le99} about two decades ago. Like the real forms of simple Lie algebras, quantum symmetric pairs can also be classified with Satake diagrams.
	
	In their significant work \cite{BW18},  Bao and Wang employed the $\imath$quantum groups of type AIII without black nodes (in the sense of Satake diagrams) to reformulate the Kazhdan-Lusztig theory of type B/C so that a conceptual solution to the problem of irreducible characters for $\mathfrak{osp}$ Lie superalgebras was provided. Such insightful application attracts considerable attention to the $\imath$quantum groups. Since then, the $\imath$-program, which is devoted to generalizing various basic constructions for quantum groups to $\imath$quantum groups, is developing rapidly and has gained a great number of achievements. Even restricted to the $\imath$quantum groups of type AIII, the $\imath$-program possesses great importance because of its close relation to various objects in the Lie theory of type B/C, such as the Lie groups/algebras, Hecke algebras, Schur algebras, flag varieties, and so on.

	\subsection{}
	Let $H$ be a Hopf algebra with counit $\varepsilon$, and let $M$ be an $H$-module algebra. The first fundamental theorem (FFT) and second fundamental theorems (SFT) of classical invariant theory provide generators and their generating relations for the subalgebra $M^H=\{v\in M~|~hv=\varepsilon(h)v,\forall h\in H\}$, respectively.
	
	If $H=\mathbb{C}GL_n(\mathbb{C})$ is the group algebra of the general linear group $GL_n(\mathbb{C})$ and $M=\mathrm{sym}(V^{\oplus r}\oplus {V^*}^{\oplus s})$ the symmetric algebra where $V=\mathbb{C}^n$ is the natural representation of $GL_n(\mathbb{C})$ and $V^*$ its dual, then the associated classical invariant theory can be traced back to the celebrated Schur duality, which describes the endomorphism algebras $\mathrm{End}_{GL_n(\mathbb{C})}(V^{\otimes d})$.
	The classical Howe duality, yielding an interplay between a pair of classical groups, also provides a representation theoretical treatment for classical invariant theory (cf. \cite{Ho89}).
	In principle for classical groups, the fundamental theorems of invariant theory, the Schur duality and the Howe duality are equivalent, though they are not equally straightforward.
	
	Now there have been many literatures focusing on quantum analogues of the Schur duality, Howe duality and fundamental theorems of invariant theory.
	For type A (i.e. quantum general linear groups $\qU_N=U_q(\mathfrak{gl}_N)$), the Schur duality, often refereed to as Schur-Jimbo duality, was established in 1986 \cite{Jim86}. The Howe duality
	$\qU_M\curvearrowright \mathcal{V}_{N,M}\curvearrowleft \qU_N$,
	where the $(\qU_M,\qU_N)$-module algebra ${\mathcal V}_{N,M}$ can be regarded as a non-commutative analogue of the symmetric algebra $\mathrm{sym}(\mathbb{C}^M\otimes{\mathbb{C}^N}^*)$, was obtained by Zhang \cite{Z03} via quantum coordinate algebras and by Baumann \cite{Ba07} via geometric Beilinson-Lusztig-MacPherson realization, independently. It was proved in \cite{LX22} that these two approaches coincide. Based on the Howe duality, Lehrer-Zhang-Zhang \cite{LZZ11} introduced a $\qU_N$-module algebra $\mathcal{P}_{P,R}=\mathcal{V}_{P,N} \otimes \overline{\mathcal{V}}_{R,N}$, where $\overline{\mathcal{V}}_{R,N}$ is a dual of $\mathcal{V}_{R,N}$ in a certain sense, and then they derived the FFT for $\qU_N$-invariants on $\mathcal{P}_{P,R}$. This $\qU_N$-module algebra $\mathcal{P}_{P,R}$ is a quantum analogue of the symmetric algebra $\mathrm{sym}((\mathbb{C}^N)^{\oplus P}\oplus ({\mathbb{C}^N}^*)^{\oplus R})$. The associated SFT was shown in \cite{Z20} where both FFT and SFT for the quantum general linear supergroups are established. See also \cite{CW20,CW22} for the Howe duality and FFT for quantum queer supergroups.
	
	The same picture for type B/C (i.e. $\imath$quantum groups of type AIII) has not been completed yet. The $\imath$Schur duality was shown in \cite{BW18} (see also \cite{BWW18, FLLLWW} for multi-parameter case). The authors introduced a $(\qU_m^\imath,\qU_n^\imath)$-module $\mathcal{V}^\imath_{n,m}$ and achieved the $\imath$Howe duality in \cite{LX22} via both Zhang's algebraic approach and Baumann's geometric approach. Similar to type A, now it is natural to ask what the fundamental theorems of the $\qU_n^\imath$-invariants on a tensor space $\mathcal{P}^\imath_{p,r}$ (an $\imath$-analogue of $\mathcal{P}_{P,R}$) are.
	
	\subsection{}
	The goal of this manuscript is to give the FFT and SFT for $\qU^\imath_n$-invariants on  $\mathcal{P}^\imath_{p,r}$. But there exist several obstacles we have to overcome.
	
	
	The $\imath$quantum group $\qU^\imath$ is not a Hopf algebra, without coproduct or antipode. Moreover, it has no R-matrix to afford the isomorphism between the tensor modules $A\otimes B$ and $B\otimes A$ (The K-matrix of an $\imath$quantum group is not applicable for this problem). But in the formulation of the FFT for $\qU_N$ in \cite{LZZ11}, all these operations are necessary. For this reason, we shall revisit the FFT and SFT for $\qU_N$ in \S~\ref{rew} where some constructions in \cite{LZZ11} are modified. Especially, we let $\mathcal{P}_{P,R}$ be $\mathcal{V}_{P,N} \otimes \mathcal{V}_{N,R}$, instead of $\mathcal{V}_{P,N} \otimes \overline{\mathcal{V}}_{R,N}$. Though such $\mathcal{P}_{P,R}$ is not only a $\qU_N$-module but also an algebra, it is no longer a $\qU_N$-module algebra. Nonetheless, it has the advantage that we can give an equivalent condition for the $\qU_N$-invariants on $\mathcal{P}_{P,R}$, which do not need to use coproduct, antipode or R-matrix. Based on this observation, we set $\mathcal{P}^\imath_{p,r}=\mathcal{V}^\imath_{p,n} \otimes \mathcal{V}^\imath_{n,r}$ and define the invariants $\mathcal{X}_{p,r}^\imath$ as \eqref{def:itensor} inspired by Lemma~\ref{lem:inv}.
	
	Another obstacle is that $\mathcal{P}^\imath_{p,r}$ is not an algebra (which is also caused by the lack of coalgebra structure for $\qU_n^\imath$). So the FFT and SFT can not be formulated in terms of  generators and generating relations of algebras. Inspired by \cite{LZZ11}, we introduce a $\mathcal{V}_{P,R}$-module homomorphism $\Psi_{p,r}^\imath: \mathcal{V}_{p,r}^\imath\rightarrow \mathcal{P}_{p,r}^\imath$, which maps onto the invariants $\mathcal{X}_{p,r}^\imath$. Thus the FFT we formulate in this paper provides generators of $\mathcal{X}^\imath_{p,r}$ as a $\mathcal{V}_{P,R}$-module, and the SFT indicates the kernel of $\Psi_{p,r}^\imath$.

	\subsection{}
	We organize the paper as follows. In Section 2, we list some basic knowledge about Hopf algebras, and recall the invariant theory for quantum general linear groups and $\imath$Howe dualities for $\imath$quantum groups of type AIII. Sections 3 is devoted to the FFT and SFT for $\imath$quantum groups of type AIII.
	
	\vspace{2mm}
	\noindent {\bf Convention.}
	Throughout the paper, we take $n\in\frac{1}{2}\mathbb{N}$ and $N\in \mathbb{N}$ with $N=2n+1$. Write $n^+=\lceil n+\half \rceil$ and $n^-=\lfloor n+\half \rfloor$.
	Let $\underline{\mathbf{n}}=\{-n,-n+1,\ldots,n-1,n\}$, $\mathbb{I}_n=\{-n+\frac{1}{2},-n+\frac{3}{2},\ldots,n-\frac{3}{2},n-\frac{1}{2}\}$ and $\mathbb{I}_n^{\imath}=\{i\in\mathbb{I}_n~|~i>0\}$. Similar notations will be used for $m,r,p\in\frac{1}{2}\mathbb{N}$ and $M, R, P\in \mathbb{N}$ with $M=2m+1, R=2r+1,P=2p+1$.
	We shall work over the field $\K=\mathbb{C}(q)$, where $q$ is an indeterminate.
	
	\vspace{2mm}
	\noindent {\bf Acknowledgement.}
	We thank Yongjie Wang for a number of helpful remarks. We would like to express our gratitude to the referees for their insightful comments towards several improvements of this article.
	The work is partially supported by the Science and Technology Commission of Shanghai Municipality (grant No. 21ZR1420000, 22DZ2229014), the NSF of China (grant No. 12371028), and Fundamental Research Funds for the Central Universities.
	
	\vspace{2mm}
	\noindent {\bf Conflict interest statement.}
	On behalf of all authors, the corresponding author states that there is no conflict of interest.
	
	\vspace{2mm}
	\noindent {\bf Data availability statement.}
	This manuscript has no associated data.

\section{Quantum groups and $\imath$quantum groups}
	\subsection{Preliminary on Hopf algebras}
	Let $H$ be a Hopf algebra (over a field $\mathbb{F}$) whose coproduct, counit and antipode are denoted by $\Delta$, $\varepsilon$ and $S$, respectively.
	For any $c \in H$, we shall omit summation symbol and denote $\Delta^{(n-1)}(c)$ simply by $c_{(1)}\otimes \cdots \otimes c_{(n)}$.

	An algebra $A$ is called a left (resp. right) \emph{$H$-module algebra} if it is a left (resp. right) $H$-module and the multiplication of $A$ is an $H$-module homomorphism from $A\otimes A$ to $A$.
	
	For any left $H$-module $A$ and right $H$-module $B$, their tensor product $A\otimes B$ is a left $H$-module whose $H$-action is as follows:
	\begin{equation}\label{def:tensor}
		c(a\otimes b)=c_{(1)} a\otimes b S(c_{(2)}), \quad(\forall c\in H, a\in A, b\in B).
	\end{equation}
	
	\begin{rem} Let $A$ and $B$ be a left and a right $H$-module algebra, respectively. Though $A\otimes B$ is both an algebra and a (left) $H$-module, it is not a (left) $H$-module algebra in general.
	\end{rem}
	
	The $H$-invariants on a left $H$-module $A$ are defined as
	\begin{equation}\label{def:inv}
		A^H:=\{a\in A~|~ca=\varepsilon(c)a, \forall c\in H\}.
	\end{equation}
	
	The following lemma might be well-known to experts. For safety, we provide a proof.
	\begin{lem}\label{lem:inv}
		Let $H$ be a Hopf algebra, and let $A$ (resp. $B$) be a left (resp. right) $H$-module algebra.
		\begin{itemize}
			\item[(a)] The space of invariants $(A\otimes B)^H$ is a subalgebra of $A \otimes B$.
			\item[(b)] Moreover, it holds that
			\begin{equation}\label{eq:inv}
				(A\otimes B)^H=\{\sum_i a_i\otimes b_i\in A\otimes B~|~\sum_ica_i\otimes b_i=\sum_ia_i\otimes b_ic, \forall c\in H\}.
			\end{equation}
		\end{itemize}
	\end{lem}
	\begin{proof}
		Take any $\sum_i a_i \otimes b_i, \sum_j a_j' \otimes b_j' \in (A\otimes B)^H$. We calculate that
		\begin{align*}
			&c(\sum_i a_i \otimes b_i)(\sum_j a_j' \otimes b_j')=\sum_{i,j}c_{(1)}(a_ia_j')\otimes (b_ib_j')S(c_{(2)})\\
			=&\sum_{i,j}(c_{(1)}a_i)(c_{(2)}a_j')\otimes (b_iS(c_{(4)}))(b_j'S(c_{(3)}))\\
			=&\big(\sum_i c_{(1)}a_i \otimes b_iS(c_{(4)})\big) \big(\sum_j c_{(2)}a_j' \otimes b_j'S(c_{(3)})\big)\\
			=&\big(\sum_i c_{(1)}a_i \otimes b_iS(c_{(3)})\big) \big(\varepsilon(c_{(2)})\sum_j a_j' \otimes b_j'\big)\\
			=&\big(\sum_i c_{(1)}a_i \otimes b_iS(c_{(2)})\big) (\sum_j a_j' \otimes b_j')\\
			=&\varepsilon(c) (\sum_i a_i \otimes b_i) (\sum_j a_j' \otimes b_j'),
		\end{align*}
		which implies that $(A\otimes B)^H$ is a subalgebra of $A\otimes B$. The statement (a) is proved.

		Denote the RHS of \eqref{eq:inv} by $\mathscr{V}$.
		For any $\sum_i a_i\otimes b_i\in \mathscr{V}$, we have
		$$c(\sum_i a_i\otimes b_i)= \sum_i c_{(1)} a_i\otimes b_i S(c_{(2)})=\sum_i c_{(1)} S(c_{(2)})a_i\otimes b_i=\varepsilon(c)\sum_i a_i\otimes b_i,\quad (\forall c\in H).$$
		Hence $\mathscr{V}\subset(A\otimes B)^H$.
		
		Conversely, we assume $\sum_i a_i\otimes b_i\in(A\otimes B)^H$. Then for any $c\in H$, we have
		\begin{align*}
			&\sum_i c_{(1)}\varepsilon(c_{(2)})a_i\otimes b_i=\sum_i c_{(1)}a_i\otimes b_i \varepsilon(c_{(2)})=\sum_i ca_i\otimes b_i=\sum_i\varepsilon(c_{(1)})a_i\otimes b_ic_{(2)}\\
			&=\sum_i c_{(1)}a_i\otimes b_i S(c_{(2)})c_{(3)}=\sum_i a_i\otimes b_i c,
		\end{align*} which implies $(A\otimes B)^H\subset\mathscr{V}$.
		
		Therefore we have $\mathscr{V}=(A\otimes B)^H$, i.e. the statement (b).
	\end{proof}
	
	A subalgebra $C$ of a Hopf algebra $H$ is called a right (resp. left) \emph{coideal subalgebra} if $\Delta(C)\subset C\otimes H$ (resp. $\Delta(C)\subset H\otimes C$).
	Let $\widetilde{A}$ (reps. $\widetilde{B}$) be a left (resp. right) $C$-module. Note that there is no natural $C$-action on $\widetilde{A}\otimes \widetilde{B}$ like \eqref{def:tensor} because of the lack of coproduct and antipode of $C$.
	Inspired by Lemma~\ref{lem:inv}(b), we can still define
	\begin{equation}\label{def:itensor}
		(\widetilde{A}\otimes \widetilde{B})^C:=\{\sum_i  a_i\otimes b_i\in \widetilde{A}\otimes \widetilde{B}~|~\sum_i ca_i\otimes b_i=\sum_i a_i\otimes b_ic, \forall c\in C\}.
	\end{equation}

	\subsection{Quantum general linear group}
	Let $\qU_N$ denote the quantum general linear group $U_q(\mathfrak{gl}_N)$
	over $\K$ with generators $E_i,F_i \ (i\in \mathbb{I}_n)$ and $D_a^{\pm 1} \ (a\in\Nid)$, subject to the following relations:
	\begin{align*}
		D_aD_a^{-1}=1,\quad &D_{a}^{\pm1}D_{b}^{\pm1}=D_{b}^{\pm1}D_{a}^{\pm1},\\
		D_aE_iD_a^{-1}=q^{\delta_{a-\frac{1}{2},i}-\delta_{a+\frac{1}{2},i}}E_i,\quad &D_aF_iD_a^{-1}=q^{\delta_{a+\frac{1}{2},i}-\delta_{a-\frac{1}{2},i}}F_i,\\
		E_iF_j-F_jE_i=&\delta_{ij}\frac{K_i-K_i^{-1}}{q-q^{-1}},\quad\mbox{where $K_i=D_{i-\frac{1}{2}}D_{i+\frac{1}{2}}^{-1}$},\\
		E_iE_j=E_jE_i,\quad &F_iF_j=F_jF_i, \quad(|i-j|>1),\\
		E_i^2E_j+E_jE_i^{2}=(q+q^{-1})E_iE_jE_i, \quad &F_i^2F_j+F_jF_i^{2}=(q+q^{-1})F_iF_jF_i,\quad(|i-j|=1),
	\end{align*}
	
	There is a Hopf algebra structure on $\qU_N$ with the coproduct $\Delta$, the
	counit $\ve$, and the antipode $S$ as follows:
	\begin{align*}
		\Delta(E_i)&=E_i \otimes K_i^{-1} +1 \otimes E_i,\quad
		\Delta(F_i)=F_i \otimes 1+K_i \otimes F_i,\quad
		\Delta(D_a)=D_a \otimes D_a;\\
		\ve(E_i)&=\ve(F_i)=0,\quad \ve(D_a)=1;\\
		S(E_i)&=-E_i K_i,\quad
		S(F_i)=-K_i^{-1} F_i,\quad
		S(D_a)=D_a^{-1}.
	\end{align*}
	
	\subsection{Quantum coordinate algebra $\mathcal{V}_{N,M}$}
	Let $\V_N$ be the natural representation of $\qU_N$, which has a standard basis $\{ v_j ~|~ j \in\Nid\}$ such that
	$$ D_a v_j=q^{\delta_{aj}} v_j, \quad E_i v_j=\delta_{i,j-\half} v_{j-1}, \quad  F_i v_j=\delta_{i,j+\half} v_{j+1},\quad(\forall a,j\in\Nid, i\in\mathbb{I}_n).$$
	Let $t_{ij}\in\qU_N^*$ $(i,j\in\Nid)$ be the matrix coefficients of $\V_N$ relative to the standard basis, i.e.,
	$$xv_j=\sum_{i\in\Nid}\langle t_{ij}, x \rangle v_i,\quad (\forall x\in\qU_N, j\in\Nid).$$
	The \emph{quantum coordinate algebra} $\mathcal{V}_{N,M}$ is a $\K$-algebra generated by $t_{ij}$ $(i\in\Nid,j\in\Mid)$ satisfying that
	for $i,k \in \underline{\mathbf{n}}$ and $j,l\in\underline{\mathbf{m}}$ with $i<k, j<l$,
	\begin{align*}
		t_{ij} t_{kl}=t_{kl} t_{ij}+(q-q^{-1})t_{il}t_{kj},\quad
		t_{il} t_{kj}=t_{kj} t_{il},\quad
		t_{ij} t_{il}=qt_{il} t_{ij},\quad
		t_{ij} t_{kj}=qt_{kj} t_{ij}.
	\end{align*}
	It can be regarded as a non-commutative analogue of the symmetric algebra $\mathrm{sym}(\mathbb{C}^M\otimes{\mathbb{C}^N}^*)$.
	Furthermore, it is shown in \cite{Ta92} that $\mathcal{V}_{N,M}$ is a $(\qU_M,\qU_N)$-module algebra, on which Zhang \cite{Z03} established the Howe duality between $\qU_M$ and $\qU_N$. Precisely, the $(\qU_M,\qU_N)$-action on $\mathcal{V}_{N,M}$ is given as follows:
	for $x\in\qU_M, y\in\qU_N, t_{ij}\in\mathcal{V}_{N,M}$,
	\begin{equation}\label{action:xtij}
		xt_{ij}=\sum_{k \in \Mid} \langle t_{kj},x\rangle t_{ik},\qquad t_{ij}y=\sum_{k \in \Nid} \langle t_{ik},y\rangle t_{kj},
	\end{equation}
	where the $t_{kj}$'s (resp. $t_{ik}$'s) in the first (resp. second) equality are matrix coefficients in $\qU_M^*$ (resp. $\qU_N^*$).
	
	For $\mi=(i_1,\ldots,i_d) \in \underline{\mathbf{n}}^d, \mj=(j_1,\ldots,j_d) \in \underline{\mathbf{m}}^d$, we always denote $$v_\mi:=v_{i_1}\otimes v_{i_2}\otimes v_{i_d}\in \mathbb{V}_N^{\otimes d} \quad\mbox{and}\quad t_{\mi\mj}:=t_{i_1 j_1} \cdots t_{i_d j_d}\in\mathcal{V}_{N,M}.$$
	For $x \in \qU_N$ and $\mi,\mj \in \Nid^d$, let
	$$x_{\mi\mj}:=\prod_{a=1}^d\langle t_{i_aj_a},x_{(a)}\rangle \in \K.$$
	By definition, we have $xv_\mj=\sum_{\mi \in \Nid^d} x_{\mi\mj} v_\mi$.
	
	\begin{lem}\label{lem:xty}
		For $x \in \qU_M$, $y \in \qU_N$, $\mi \in \Nid^d$, $\mj \in \Mid^d$, we have
		\begin{align*}
			xt_{\mi\mj}&=\sum_{\mk \in \Mid^d} x_{\mk\mj}t_{\mi\mk}, \quad t_{\mi\mj}y=\sum_{\mk \in \Nid^d} y_{\mi\mk}t_{\mk\mj}.
		\end{align*}
	\end{lem}
	
	\begin{proof}
		We shall only prove the first equality since the other is similar.
		It is a straightforward computation that
		\begin{align*}
			xt_{\mi\mj}&=(x_{(1)}t_{i_1j_1})\cdots (x_{(d)}t_{i_dj_d})=\sum_{\mk \in \Mid^d} (\langle t_{k_1j_1}, x_{(1)}\rangle t_{i_1k_1})\cdots (\langle t_{k_dj_d},x_{(d)}\rangle t_{i_dk_d})\\
			&=\sum_{\mk \in \Mid^d} (\prod_{a}\langle t_{k_aj_a}, x_{(a)}\rangle) t_{\mi\mk}=\sum_{k \in \Mid^d} x_{\mk\mj} t_{\mi\mk}.
		\end{align*}
	\end{proof}
	
	In \cite{LZZ11}, another quantum coordinate algebra $\overline{\mathcal{V}}_{N,M}$ is introduced in a similar way via replacing $t_{ij}$ by $\overline{t}_{ij}$, which are the matrix coefficients of the dual of the natural representation.
	
	\subsection{Review of the FFT and the SFT for $\qU_N$}\label{rew}
	The FFT and the SFT for $\qU_N$-invariants on $\mathcal{V}_{P,N} \otimes \overline{\mathcal{V}}_{R,N}$ are given in \cite{LZZ11} and \cite{Z20}, respectively. In their constructions, the $R$-matrix is employed to introduce a twisted multiplication on the tensor space $\mathcal{V}_{P,N} \otimes \overline{\mathcal{V}}_{R,N}$ so that $\mathcal{V}_{P,N} \otimes \overline{\mathcal{V}}_{R,N}$ is a $\qU_N$-module algebra.
	
	Below we shall consider the $\qU_N$-invariants on $\mathcal{V}_{P,N} \otimes \mathcal{V}_{N,R}$ instead of $\mathcal{V}_{P,N} \otimes \overline{\mathcal{V}}_{R,N}$. We must acknowledge that $\mathcal{V}_{P,N} \otimes \mathcal{V}_{N,R}$ is no longer a $\qU_N$-module algebra. This weakness is acceptable under our motivation since the $\imath$-analogue of $\mathcal{V}_{N,M}$ is no longer an algebra (and hence the associated tensor space cannot be a module algebra).
	
	Denote by $\mathcal{P}_{P,R}:=\mathcal{V}_{P,N} \otimes \mathcal{V}_{N,R}$ the tensor algebra of $\mathcal{V}_{P,N}$ and $\mathcal{V}_{N,R}$. As mentioned in \eqref{def:tensor}, the tensor algebra $\mathcal{P}_{P,R}$ is still a $\qU_N$-module defined by
	$$x(f\otimes g)=(x_{(1)}f)\otimes (gS(x_{(2)})),\qquad (\forall x\in \qU_N, f\in \mathcal{V}_{P,N}, g\in\mathcal{V}_{N,R}).$$
	Let $\mathcal{X}_{P,R}:=\mathcal{P}_{P,R}^{\qU_N}$ denote the space of $\qU_N$-invariants in $\mathcal{P}_{P,R}$. By definition and Lemma~\ref{lem:inv},
	\begin{align*}
		\mathcal{X}_{P,R}&=\{\sum f\otimes g\in\mathcal{P}_{P,R}~|~x(\sum f\otimes g)=\varepsilon(x)(\sum f\otimes g), \forall x\in \qU_N\}\\
		&=\{\sum f\otimes g\in\mathcal{P}_{P,R}~|~ \sum xf\otimes g=\sum f \otimes gx, \forall x\in \qU_N\}.
	\end{align*}
	Both $\mathcal{P}_{P,R}$ and $\mathcal{X}_{P,R}$ are $(\qU_R,\qU_P)$-modules.
	
	Introduce the following elements in $\mathcal{P}_{P,R}$:
	\begin{align*}
		X_{ij}:=\sum_{k \in\Nid} t_{ik} \otimes t_{kj}, \quad (i \in \underline{\mathbf{p}}, j \in \underline{\mathbf{r}}).
	\end{align*}
	By \eqref{action:xtij}, we have $X_{ij}\in\mathcal{X}_{P,R}$.
	
	For any $d\in\mathbb{N}$,
	denote
	\begin{align*}
		\mathfrak{A}_{d,M}:=\{\mi=(i_1,\ldots,i_d)\in \Mid^d ~|~ -m \leq i_1 < \cdots < i_d \leq m\}.
	\end{align*}
	For $\mi \in \mathfrak{A}_{N,P}$ and $\mj \in \mathfrak{A}_{N,R}$, the \emph{quantum $N$-minor} $\Delta(\mi,\mj) \in \mathcal{V}_{P,R}$ is defined as
	\begin{align*}
		\Delta(\mi,\mj):=\sum_{w \in \mathfrak{S}_N} (-q)^{\ell(w)}t_{\mi, \mj w},
	\end{align*}
	where $\ell(w)$ is the length of $w$, and $\mj w$ means the natural right $\mathfrak{S}_N$-action on $\underline{\mathbf{r}}^N$.
	
	Though we slightly adjust the constructions used in \cite{LZZ11} and \cite{Z20}, it follows from Lemma~\ref{lem:inv}(a) that $\mathcal{X}_{P,R}$ is still a subalgebra of $\mathcal{P}_{P,R}$. Moreover, the arguments in \cite{LZZ11} and \cite{Z20} are still valid to obtain the following theorem on invariant theory of $\qU_N$.
	\begin{thm}[\cite{LZZ11,Z20}] \label{UN:FFTandSFT}
		{\em(1) \textbf{FFT for $\qU_N$}}. There exists a surjective algebra homomorphism $\Psi_{P,R}: \mathcal{V}_{P,R} \rightarrow \mathcal{X}_{P,R}$ defined by
		$\Psi_{P,R}(t_{ij})= X_{ij}, (\forall i\in\underline{\mathbf{p}}, j\in\underline{\mathbf{r}})$,
		which means that $\{X_{ij}~|~i\in\underline{\mathbf{p}}, j\in\underline{\mathbf{r}}\}$ generates the $\qU_N$-invariant subalgebra $\mathcal{X}_{P,R}$.\\
		{\em(2) \textbf{SFT for $\qU_N$}}.
		The ideal $\Ker\Psi_{P,R}$ of $\mathcal{V}_{P,R}$ is generated by quantum $N$-minors, which offers the generating relations of generators $X_{ij}, (i\in\underline{\mathbf{p}}, j\in\underline{\mathbf{r}})$. Moreover, the set $\{ \Delta(\mi,\mj) ~|~ \mi \in \mathfrak{A}_{N,P}, \mj \in \mathfrak{A}_{N,R}\}$ generates  $\Ker\Psi_{P,R}$ as a two-sided ideal of $\mathcal{V}_{P,R}$.
	\end{thm}
	
	\subsection{$\imath$quantum group $\qU_n^\imath$}
	The $\imath$quantum group $\qU_n^\imath$ (of quasi-split type AIII in the sense of Satake diagram) is the subalgebra of $\qU_N$ generated by
	\begin{align*}
		e_i&=E_{i}+K_i^{-1}F_{-i},\quad f_i=E_{-i}+F_{i}K_{-i}^{-1} \quad (i\in \mathbb{I}_n), \quad
		d_a=D_aD_{-a}\quad (0\leq a\in\Nid),
	\end{align*} and in the case of $n\in\frac{1}{2}+\N$, together with
	\begin{align*}
		t&=E_0+qF_0K_0^{-1}+K_0^{-1}\quad\mbox{ (Note: $t$ does not exist if $n\in\N$)}.
	\end{align*}
	It can be checked directly that $\qU_n^\imath$ is a right coideal subalgebra of $\qU_N$, i.e., $\Delta(\qU_n^\imath)\subset \qU_n^\imath\otimes \qU_N$. The standard $\qU_N$-module $\mathbb{V}_N$ is also a $\qU_n^\imath$-module.

	At the specialization $q\rightarrow1$, the $\imath$quantum group $\qU_n^\imath$ becomes to the universal enveloping algebra $U(\mathfrak{gl}_{n^+}\oplus\mathfrak{gl}_{{n^-}})$.
	Let $\Par_{n,d}$ be the set of partitions of $d$ with at most $n$ parts, and let $\Par_{n,d}^\imath:= \bigsqcup_l \Par_{n^+,l} \times \Par_{n^-,d-l}$ be the set of bi-partitions of $d$ with at most $(n^+,n^-)$ parts. Denote $\Par_{n}^\imath:=\bigsqcup_{d}\Par_{n,d}^\imath$. For $\lambda\in\Par_{n}^\imath$, Watanabe \cite{Wa21} introduced the left (resp. right) irreducible highest weight $\qU_n^\imath$-module $L^{n,\imath}_\lambda$ (resp. $L^{n,\imath*}_\lambda$), whose classical limit is the irreducible $(\mathfrak{gl}_{n^+}\oplus\mathfrak{gl}_{{n^-}})$-module with highest weight $\lambda$.
	
\begin{rem}
The $\imath$quantum group $\qU^\imath$ (of quasi-split type AIII) comes with parameters (see \cite{Le02}). But for our purpose (telling a story among the Schur duality, the Howe duality and the invariant theory), the parameters are fixed once for all by its double centralizer property with Hecke algebras of type B (i.e. the $\imath$Schur duality which we shall mention in the next subsection).
\end{rem}
	
	\subsection{$\imath$Schur duality}
	Let $s_0,s_1\ldots,s_{d-1}$ be the simple reflections of the Weyl group $W_{B_d}$ (of type $B_d$). There is a right $W_{B_d}$-action on $\Nid^d$ (also $\Mid^d$, etc.) given by
	\begin{align*}
		\mi s_0=(-i_1,i_2,\ldots,i_d)\qquad\mbox{and}\qquad \mi s_a=(\ldots,i_{a+1},i_a,\ldots)\quad (\forall a=1,2,\ldots,d-1).
	\end{align*}

	The \emph{Hecke algebra} $\HH_{B_d}$ (of type $B_d$) is a $\mathbb{K}$-algebra generated by $T_0, T_1,\ldots,T_{d-1}$ subject to
	\begin{align*}
		&(T_i-q^{-1})(T_i+q)=0,\\
		&T_{i}T_{i+1}T_i=T_{i+1}T_iT_{i+1}\ (i\neq0),\\
		&(T_0T_1)^2=(T_1T_0)^2,\quad T_iT_j=T_jT_i \ (|i-j|>1).
	\end{align*}
	For any $w\in W_{B_d}$ with a reduced form $w=s_{i_1}s_{i_2}\cdots s_{i_\ell}$, denote $T_w:=T_{i_1}T_{i_2}\cdots T_{i_\ell}$.
	
	It is known that $\mathbb{V}_N^{\otimes d}$ admits an $\HH_{B_d}$-action defined as follows:
	\begin{align*}
		v_{\mi}T_0&=\left\{
		\begin{array}{ll}
			v_{\mi s_0}, &\mbox{if $i_1>0$};\\
			q^{-1}v_{\mi s_0}, &\mbox{if $i_1=0$};\\
			v_{\mi s_0}+(q^{-1}-q)v_{\mi}, &\mbox{if $i_1<0$},
		\end{array}
		\right. \quad
		v_{\mi}T_a=\left\{
		\begin{array}{ll}
			v_{\mi s_a}, &\mbox{if $i_a<i_{a+1}$};\\
			q^{-1}v_{\mi s_a}, &\mbox{if $i_a=i_{a+1}$};\\
			v_{\mi s_a}+(q^{-1}-q)v_{\mi}, &\mbox{if $i_a>i_{a+1}$}.
		\end{array}
		\right.
	\end{align*}
	The $\imath$Schur duality (i.e. quantum Schur duality of type B) says that $\qU_n^\imath$ and $\HH_{B_d}$ admit a double centralizer property on $\mathbb{V}^{\otimes d}$ (see \cite{BW18, BWW18}).
	\subsection{$\imath$Howe duality}\label{iH}
	
	Let $\mathcal{V}_{n,m}^\imath=\mathcal{V}_{N,M}/\mathcal{I}_{n,m}$
	where $\mathcal{I}_{n,m}$ is the right ideal of $\mathcal{V}_{N,M}$ generated by, $(0<i \in \Nid, 0<j \in \Mid)$,
\begin{equation}\label{rel:tij}
\begin{array}{l}
t_{ij}-t_{-i,-j}+(q^{-1}-q)t_{i,-j}, \quad t_{i,-j}-t_{-i,j},\qquad \mbox{and}\\
t_{i,0}-qt_{-i,0} \ \mbox{(in the case of $m\in\mathbb{Z}$)},\quad  t_{0,j}-qt_{0,-j}\  \mbox{(in the case of $n\in \mathbb{Z}$)}.
\end{array}
\end{equation}
We denote the image of $f\in \mathcal{V}_{N,M}$ in $\mathcal{V}_{n,m}^\imath$ by $\widetilde{f}$.
It holds that $x\widetilde{f}y=\widetilde{xfy}$ for any $f \in \mathcal{V}_{N,M}, x \in \qU_m^\imath, y \in \qU_n^\imath$. When $n=m$, the right ideal $\mathcal{I}_{n,n}$ consists of the functions annihilating $\qU_n^\imath$, which coincides with the one introduced in \cite[Proposition~2.4.4]{LNX22}.
	
At the classical limit $q \to 1$, the quantum coordinate algebra $\mathcal{V}_{N,M}$ specializes to $\fC[\mathrm{Mat}_{N \times M}]$ and the right ideal $\mathcal{I}_{n,m}$ specializes to the ideal associated with the subvariety $\mathrm{Mat}_{N \times M}^\theta:= \{(a_{ij}) \in \mathrm{Mat}_{N \times M} ~|~ a_{ij}=a_{-i,-j}\}$. Therefore, the quotient $\mathcal{V}_{n,m}^\imath$ specializes to the algebra $$\fC[\mathrm{Mat}_{N \times M}^\theta] \simeq \fC[\mathrm{Mat}_{n^+ \times m^+}] \oplus \fC[\mathrm{Mat}_{n^- \times m^-}].$$
	
The following lemma, which will be employed in the proof of Lemma~\ref{dual},  can be computed by \eqref{rel:tij} directly.
	\begin{lem}\label{com}
		For $\mi \in \Nid^d, \mj \in \Mid^d$ and $0< a< d$, we have
$$\widetilde{t}_{\mi s_0,\mj s_0}=
		\begin{cases}
			\widetilde{t}_{\mi\mj}+(q-q^{-1})\widetilde{t}_{\mi,\mj s_0} & \mbox{if $i_1,j_1>0$},\\
			\widetilde{t}_{\mi\mj} & \mbox{if $i_1>0>j_1$},\\
			{q^{-1}}\widetilde{t}_{\mi\mj} & \mbox{if $i_1>0=j_1$ or $i_1=0<j_1$;}
		\end{cases}$$
		$$\widetilde{t}_{\mi s_a,\mj s_a}=
		\begin{cases}
			\widetilde{t}_{\mi\mj}+(q-q^{-1})\widetilde{t}_{\mi,\mj s_a} & \mbox{if $i_a<i_{a+1}, j_a<i_{a+1}$},\\
			\widetilde{t}_{\mi\mj} & \mbox{if $i_a<i_{a+1}, j_{a+1}<j_a$},\\
			{q^{-1}}\widetilde{t}_{\mi\mj} & \mbox{if $i_a<i_{a+1}, j_a=j_{a+1}$ or $i_a=i_{a+1}, j_a<j_{a+1}$.}
		\end{cases}$$
	\end{lem}
	
	We define $\mathcal{V}_{n,m}^\imath$ in \cite{LX22} by a different but essentially equivalent way. It has been shown in \cite[Proposition~5.5]{LX22} that $\mathcal{V}_{n,m}^\imath$ is a $(\qU_m^\imath, \qU_n^\imath)$-module. Following is a direct corollary of Lemma~\ref{lem:xty}.
	\begin{lem}\label{iope}
		Let $x \in \qU_m^\imath$, $y \in \qU_n^\imath$, $\mi \in \Nid^d$, $\mj \in \Mid^d$. We have
		\begin{align*}
			x\widetilde{t}_{\mi\mj}&=\sum_{\mk \in \Mid^d} x_{\mk\mj}\widetilde{t}_{\mi\mk}, \quad \widetilde{t}_{\mi\mj}y=\sum_{\mk \in \Nid^d} y_{\mi\mk}\widetilde{t}_{\mk\mj}.
		\end{align*}
	\end{lem}
	
	Noting that $\mathcal{V}_{N,M}$ is an $\N$-graded algebra (by setting $\mathrm{deg}(t_{ij})=1$) and $\mathcal{I}_{n,m}$ is a homogeneous right ideal of $\mathcal{V}_{N,M}$. Hence $\mathcal{V}_{n,m}^\imath$ is an $\N$-graded $\mathcal{V}_{N,M}$-module, whose homogeneous component of degree $d$ is denoted by $\mathcal{V}_{n,m,d}^\imath$.
	
	It was showed in \cite[Proposition~5.5]{LX22} that $\mathcal{V}_{n,m}^\imath$ is a $(\qU^\imath_m,\qU^\imath_n)$-module. In fact, the formulas therein show that its homogeneous component $\mathcal{V}_{n,m,d}^\imath$ is a $(\qU^\imath_m,\qU^\imath_n)$-module, too. Moreover, it was proved in \cite[Theorem~5.7]{LX22} that as $(\qU^\imath_m,\qU^\imath_n)$-modules, \begin{equation}\label{eq:iso}
		\mathcal{V}_{n,m,d}^\imath\simeq\Hom_{\HH_{B_d}}(\mathbb{V}_N^{\otimes d},\mathbb{V}_M^{\otimes d}).
	\end{equation}
	Following is the $\imath$Howe duality.
	\begin{thm}\cite[Theorem~6.4]{LX22}
		The $(\qU_m^\imath,\qU_n^\imath)$-module $\mathcal{V}_{n,m,d}^\imath$ admits the following multiplicity-free decomposition:
		\begin{align}\label{eq:iHowe}
			\mathcal{V}_{n,m,d}^\imath \simeq\Hom_{\HH_{B_d}}(\mathbb{V}_N^{\otimes d},\mathbb{V}_M^{\otimes d})\simeq \bigoplus_{\lambda \in \Par_{\min\{m,n\},d}^\imath} L_\lambda^{m,\imath} \otimes L_\lambda^{n,\imath*}.
		\end{align}
	\end{thm}

	\subsection{The dual space $\mathcal{V}_{n,m,d}^{\imath *}$}\label{subsecV}
	
	Let $\mathcal{V}_{n,m,d}^{\imath *}$ denote the dual space of $\mathcal{V}_{n,m,d}^{\imath}$, which is naturally equipped with a $(\qU_n^\imath,\qU_m^\imath)$-module structure. It follows from \eqref{eq:iHowe} that as $(\qU_n^\imath,\qU_m^\imath)$-modules,
	\begin{align}\label{dec:dual}
		\mathcal{V}_{n,m,d}^{\imath *}\simeq\Hom_{\HH_{B_d}}(\mathbb{V}_M^{\otimes d},\mathbb{V}_N^{\otimes d})\simeq \bigoplus_{\lambda \in \Par_{\min\{m,n\},d}^\imath} L_\lambda^{n,\imath} \otimes L_\lambda^{m,\imath*}.
	\end{align}
	Below we shall give the isomorphism explicitly.
	
	Let $\rho_{n,m,d}$ be the linear map $$\rho_{n,m,d}: \mathcal{V}_{n,m,d}^{\imath *} \rightarrow \Hom_\K(\mathbb{V}_M^{\otimes d}, \mathbb{V}_N^{\otimes d})\quad\mbox{via}\quad \rho_{n,m,d}(\alpha)v_\mj=\sum_{\mi \in \Nid^d} \langle \alpha,\widetilde{t}_{\mi\mj} \rangle v_\mi.$$

	\begin{lem}\label{dual}
		The map $\rho_{n,m,d}$ is an injective $(\qU_n^\imath,\qU_m^\imath)$-module homomorphism with $\mathrm{Im}(\rho_{n,m,d})=\Hom_{\HH_{B_d}}(\mathbb{V}_M^{\otimes d}, \mathbb{V}_N^{\otimes d})$.
	\end{lem}
	\begin{proof}
		Let $\alpha \in \mathcal{V}_{n,m,d}^{\imath*}$ be a nonzero element. There exists $\mi \in \Nid^d$ and $\mj \in \Mid^d$ such that $\langle \alpha, \widetilde{t}_{\mi\mj}\rangle \neq 0$ and hence $\rho_{n,m,d}(\alpha) v_\mj \neq 0$. Thus $\rho_{n,m,d}$ is injective.
		
		For any $\mj \in \Mid^d$, $x \in \qU_n^\imath$ and $y \in \qU_m^\imath$, we compute
		\begin{align*}
			\rho_{n,m,d}(x\alpha y)v_\mj&=\sum_{\mi \in \Nid^d} \langle x\alpha y,\widetilde{t}_{\mi\mj} \rangle v_\mi=\sum_{\mi \in \Nid^d} \langle \alpha,y\widetilde{t}_{\mi\mj} x \rangle v_\mi=\sum_{\mi,\mk \in \Nid^d, \ml \in \Mid^d} \langle \alpha,x_{\mi\mk}y_{\ml\mj}\widetilde{t}_{\mk\ml} \rangle v_\mi\\&=\sum_{\mi,\mk \in \Nid^d, \ml \in \Mid^d} x_{\mi\mk}y_{\ml\mj}\langle \alpha,\widetilde{t}_{\mk\ml}\rangle v_\mi=\sum_{\mk\in\Nid^d,\ml\in\Mid^d}y_{\ml\mj}\langle \alpha,\widetilde{t}_{\mk\ml}\rangle x v_\mk\\&=\sum_{\ml\in\Mid^d}y_{\ml\mj}x\left(\rho_{n,m,d}(\alpha)v_\ml\right)=
			x\left(\rho_{n,m,d}(\alpha)(yv_\mj)\right),
		\end{align*}
		where the third equality holds by Lemma~\ref{iope}.
		So $\rho_{n,m,d}$ is a $(\qU_n^\imath,\qU_m^\imath)$-module homomorphism.
		
		Now let us check the statement $\mathrm{Im}(\rho_{n,m,d})=\Hom_{\HH_{B_d}}(\mathbb{V}_M^{\otimes d}, \mathbb{V}_N^{\otimes d})$.
		
		If $\mj <\mj s_a$, then
		\begin{align*}
			&(\rho_{n,m,d}(\alpha)v_\mj)T_a=\sum_{\mi \in \Nid^d} \langle \alpha,\widetilde{t}_{\mi\mj} \rangle v_\mi T_a\\
			=&\sum_{\mi \in \Nid^d, \mi<\mi s_a} \langle \alpha,\widetilde{t}_{\mi\mj} \rangle v_{\mi s_a}+\sum_{\mi \in \Nid^d, \mi=\mi s_a} q^{-1}\langle \alpha,\widetilde{t}_{\mi\mj} \rangle v_\mi+\sum_{\mi \in \Nid^d, \mi>\mi s_a} \langle \alpha,\widetilde{t}_{\mi\mj} \rangle (v_{\mi s_a}+(q^{-1}-q)v_\mi)\\
			=&\sum_{\mi \in \Nid^d, \mi<\mi s_a} \langle \alpha,\widetilde{t}_{\mi,\mj s_a} \rangle v_\mi+\sum_{\mi \in \Nid^d, \mi=\mi s_a} \langle \alpha,\widetilde{t}_{\mi,\mj s_a} \rangle v_\mi+\sum_{\mi \in \Nid^d, \mi<\mi s_a} (\langle \alpha,(q^{-1}-q)\widetilde{t}_{\mi\mj} +\widetilde{t}_{\mi s_a,\mj} \rangle) v_\mi\\
			=&\sum_{\mi \in \Nid^d} \langle \alpha,\widetilde{t}_{\mi,\mj s_a} \rangle v_\mi=\rho_{n,m,d}(\alpha)(v_\mj T_a).
		\end{align*}
		If $\mj =\mj s_a$, then
		\begin{align*}
			&(\rho_{n,m,d}(\alpha)v_\mj)T_a=\sum_{\mi \in \Nid^d} \langle \alpha,\widetilde{t}_{\mi\mj} \rangle v_\mi T_a\\
			=&\sum_{\mi \in \Nid^d, \mi<\mi s_a} \langle \alpha,\widetilde{t}_{\mi\mj} \rangle v_{\mi s_a}+\sum_{\mi \in \Nid^d, \mi=\mi s_a} q^{-1}\langle \alpha,\widetilde{t}_{\mi\mj} \rangle v_\mi+\sum_{\mi \in \Nid^d, \mi>\mi s_a} \langle \alpha,\widetilde{t}_{\mi\mj} \rangle (v_{\mi s_a}+(q^{-1}-q)v_\mi)\\
			=&\sum_{\mi \in \Nid^d} q^{-1}\langle \alpha,\widetilde{t}_{\mi\mj} \rangle v_\mi=(\rho_{n,m,d}(\alpha)v_\mj T_a).
		\end{align*}
		If $\mj >\mj s_a$, then
		\begin{align*}
			&(\rho_{n,m,d}(\alpha)v_\mj)T_a=(\rho_{n,m,d}(\alpha)v_{\mj s_a} T_a)T_a=(\rho_{n,m,d}(\alpha)v_{\mj s_a})T_a^2\\
			=&(\rho_{n,m,d}(\alpha)v_{\mj s_a})\left((q^{-1}-q)T_a +1\right)=\rho_{n,m,d}(\alpha)\left(v_{\mj s_a}((q^{-1}-q)T_a +1)\right)\\
			=&\rho_{n,m,d}(\alpha)(v_\mj T_a).
		\end{align*}
		Therefore, the image $\mathrm{Im}(\rho_{n,m,d})\subset \Hom_{\HH_{B_d}}(\mathbb{V}_M^{\otimes d}, \mathbb{V}_N^{\otimes d})$. The isomorphism \eqref{eq:iso} tells us $$\dim(\mathcal{V}_{n,m,d}^{\imath*})=\dim(\mathcal{V}_{n,m,d}^{\imath})=\dim(\Hom_{\HH_{B_d}}(\mathbb{V}_N^{\otimes d}, \mathbb{V}_M^{\otimes d}))=\dim(\Hom_{\HH_{B_d}}(\mathbb{V}_M^{\otimes d}, \mathbb{V}_N^{\otimes d})).$$ Thanks to the injectivity of $\rho_{n,m,d}$, we get $\mathrm{Im}(\rho_{n,m,d})=\Hom_{\HH_{B_d}}(\mathbb{V}_M^{\otimes d}, \mathbb{V}_N^{\otimes d})$.
	\end{proof}
	
	
\section{Invariant theory for $\qU_n^\imath$}
	\subsection{$\qU_n^\imath$-invariants}
	Let $\mathcal{P}_{p,r}^\imath:=\mathcal{V}_{p,n}^\imath \otimes \mathcal{V}_{n,r}^\imath$, which is a $(\qU_n^\imath \otimes \qU_r^\imath, \qU_p^\imath \otimes \qU_n^\imath)$-module.
	As defined in \eqref{def:itensor}, we denote the $\qU_n^\imath$-invariant spaces by
	\begin{align*}
		\mathcal{X}_{p,r}^\imath&:=\{ \sum f \otimes g \in \mathcal{P}_{p,r}^\imath ~|~ \sum (xf) \otimes g=\sum f \otimes (gx), \forall x \in \qU_n^\imath\},\\
		\mathcal{X}_{p,r,d}^\imath&:=\{ \sum f \otimes g \in \mathcal{V}_{p,n,d}^\imath \otimes \mathcal{V}_{n,r,d}^\imath ~|~ \sum (xf) \otimes g=\sum f \otimes (gx), \forall x \in \qU_n^\imath\}.
	\end{align*}

	\begin{lem}\label{fd}
		The invariant spaces $\mathcal{X}_{p,r}^\imath$ and $\mathcal{X}_{p,r,d}^\imath$ admit the following multiplicity-free decompositions of $(\qU_r^\imath,\qU_p^\imath)$-modules:
		\begin{align*}
			\mathcal{X}_{p,r}^\imath \simeq \bigoplus_{\lambda \in \Par_{\min\{n,p,r\}}^\imath} L_\lambda^{r,\imath} \otimes L_\lambda^{p,\imath*},\qquad
			\mathcal{X}_{p,r,d}^\imath \simeq \bigoplus_{\lambda \in \Par_{\min\{n,p,r\},d}^\imath} L_\lambda^{r,\imath} \otimes L_\lambda^{p,\imath*}.
		\end{align*}
		As a consequence, $\mathcal{X}_{p,r}^\imath \simeq \bigoplus_{d=0}^\infty\mathcal{X}_{p,r,d}^\imath$.
	\end{lem}
	\begin{proof}
		The proof of the decomposition of $\mathcal{X}_{p,r,d}^\imath$ is similar to the one of $\mathcal{X}_{p,r}^\imath$, so we just prove the decomposition for $\mathcal{X}_{p,r}^\imath$.
		By $\imath$Howe duality, as a $(\qU_m^\imath,\qU_n^\imath)$-module,
		$$\mathcal{V}_{n,m}^\imath \simeq \bigoplus_{\lambda \in \Par_{\min\{m,n\}}^\imath} L_\lambda^{m,\imath} \otimes L_\lambda^{n,\imath*}.$$
		So as a $(\qU_n^\imath \otimes \qU_r^\imath,\qU_p^\imath \otimes \qU_n^\imath)$-module,
		\begin{align*}
			\mathcal{P}_{p,r}^\imath &\simeq \bigoplus_{\lambda \in \Par_{\min\{n,p\}}^\imath, \mu \in \Par_{\min\{n,r\}}^\imath} L_\lambda^{n,\imath} \otimes L_\lambda^{p,\imath*} \otimes L_\mu^{r,\imath} \otimes L_\mu^{n,\imath*}\\
			&\simeq \bigoplus_{\lambda \in \Par_{\min\{n,p\}}^\imath, \mu \in \Par_{\min\{n,r\}}^\imath} \Hom_{\mathbb{K}}(L_\mu^{n,\imath}, L_\lambda^{n,\imath}) \otimes( L_\mu^{r,\imath} \otimes L_\lambda^{p,\imath*}).
		\end{align*}
		Under the above isomorphism, the condition in the definition of $\mathcal{X}_{p,r}^\imath$ reads $x\psi=\psi x$ for any $x\in \qU_n^\imath$ and $\psi\in \Hom(L_\mu^{n,\imath}, L_\lambda^{n,\imath})$, i.e.,
		\begin{align*}
			\mathcal{X}_{p,r}^\imath &\simeq \bigoplus_{\lambda \in \Par_{\min\{n,p\}}^\imath, \mu \in \Par_{\min\{n,r\}}^\imath} \Hom_{\qU_n^\imath}(L_\mu^{n,\imath}, L_\lambda^{n,\imath}) \otimes (L_\mu^{r,\imath} \otimes L_\lambda^{p,\imath*})\\
			&\simeq \bigoplus_{\lambda \in \Par_{\min\{n,p,r\}}^\imath} L_\lambda^{r,\imath} \otimes L_\lambda^{p,\imath*}.
		\end{align*}
	\end{proof}
	
	\subsection{The composition map $\Phi_{p,r,d}$}
	Recall $\mathcal{V}_{n,m,d}^{\imath*}\simeq\Hom_{\HH_{B_d}}(\V_M^{\otimes d},\V_N^{\otimes d})$. We shall identify an element in $\mathcal{V}_{n,m,d}^{\imath*}$ by its image under $\rho_{n,m,d}$, which is an $\HH_{B_d}$-homomorphism in $\Hom_{\HH_{B_d}}(\V_M^{\otimes d},\V_N^{\otimes d})$ by Lemma~\ref{dual}. There exists a natural composition map
	\begin{align*}
		\Phi_{p,r,d}: \mathcal{V}_{p,n,d}^{\imath*} \otimes \mathcal{V}_{n,r,d}^{\imath*} \rightarrow \mathcal{V}_{p,r,d}^{\imath*}, \quad \alpha \otimes \beta \mapsto \alpha \circ \beta.
	\end{align*}

	\begin{lem}\label{im}
		The morphism $\Phi_{p,r,d}$ factors through $\mathcal{V}_{p,n,d}^{\imath*} \otimes_{\qU_n^\imath} \mathcal{V}_{n,r,d}^{\imath*}$. Moreover, this induces a $(\qU_p^\imath,\qU_r^\imath)$-module isomorphism from $\mathcal{V}_{p,n,d}^{\imath*} \otimes_{\qU_n^\imath} \mathcal{V}_{n,r,d}^{\imath*}$ to $\mathrm{Im}\Phi_{p,r,d}$.
	\end{lem}
	
	\begin{proof}
		As $\rho_{n,m,d}$ is a $(\qU_n^\imath,\qU_m^\imath)$-module homomorphism by Lemma~\ref{dual}, we know $$\{\alpha x \otimes \beta -\alpha \otimes x\beta ~|~ x\in \qU_n^\imath, \alpha \in \mathcal{V}_{p,n,d}^{\imath*}, \beta \in \mathcal{V}_{n,r,d}^{\imath*}\} \subset \Ker\Phi_{p,r,d}.$$ Therefore, $\Phi_{p,r,d}$ factors through $\mathcal{V}_{p,n,d}^{\imath*} \otimes_{\qU_n^\imath} \mathcal{V}_{n,r,d}^{\imath*}$.
		
We have the following commutative diagram of $(\qU_p^\imath,\qU_r^\imath)$-modules:
		$$\begin{tikzcd}
			{\mathcal{V}_{p,n,d}^{\imath*} \otimes \mathcal{V}_{n,r,d}^{\imath*}} \arrow[d, "\cong"] \arrow[r, "\Phi_{p,r,d}"] & {\mathcal{V}_{p,r,d}^{\imath*}} \arrow[d, "\cong"]\\
			{\Hom_{\HH_{B_d}}(\V_N^{\otimes d}, \V_P^{\otimes d}) \otimes \Hom_{\HH_{B_d}}(\V_R^{\otimes d}, \V_N^{\otimes d})} \arrow[d, "\cong"] \arrow[r,, "composition"] & {\Hom_{\HH_{B_d}}(\V_R^{\otimes d}, \V_P^{\otimes d})} \arrow[d, "\cong"]\\
			{\bigoplus\limits_{\tiny \begin{array}{l}
						\lambda \in \Par_{\min\{n,p\},d}^\imath \\ \mu \in \Par_{\min\{n,r\},d}^\imath
				\end{array}} \Hom(L_\lambda^{n,\imath}, L_\lambda^{p,\imath}) \otimes \Hom(L_\lambda^{r,\imath}, L_\lambda^{n,\imath})} \arrow[r] & {\bigoplus\limits_{\lambda \in \Par_{\min\{p,r\},d}^\imath}  \Hom(L_\lambda^{r,\imath}, L_\lambda^{p,\imath})},
		\end{tikzcd}$$
		Here the commutativity of the upper half diagram follows from the definition of $\Phi_{p,r,d}$, while the commutativity of the down half follows from the Schur duality of type B. Moreover,  the bottom arrow maps onto $$\bigoplus\limits_{\lambda \in \Par_{\min\{n,p,r\},d}^\imath} \Hom(L_\lambda^{r,\imath}, L_\lambda^{p,\imath}) \cong \bigoplus\limits_{\lambda \in \Par_{\min\{n,p,r\},d}^\imath}  L_\lambda^{p,\imath} \otimes L_\lambda^{r,\imath*}.$$
On the other hand,
		\begin{align*}
			\mathcal{V}_{p,n,d}^{\imath*} \otimes_{\qU_n^\imath} \mathcal{V}_{n,r,d}^{\imath*} &\simeq \bigoplus\limits_{\tiny \begin{array}{l}
\lambda \in \Par_{\min\{n,p\},d}^\imath \\ \mu \in \Par_{\min\{n,r\},d}^\imath
\end{array}} L_\lambda^{p,\imath} \otimes  (L_\lambda^{n,\imath*} \otimes_{\qU_n^\imath} L_\mu^{n,\imath}) \otimes  L_\mu^{r,\imath*}\\
			&\simeq \bigoplus\limits_{\lambda \in \Par_{\min\{n,p,r\},d}^\imath}  L_\lambda^{p,\imath} \otimes L_\lambda^{r,\imath*}.
		\end{align*}
		Therefore, it must be a $(\qU_p^\imath,\qU_r^\imath)$-module isomorphism from $\mathcal{V}_{p,n,d}^{\imath*} \otimes_{\qU_n^\imath} \mathcal{V}_{n,r,d}^{\imath*}$ to $\mathrm{Im}\Phi_{p,r,d}$.
	\end{proof}
	
	\subsection{The FFT for $\qU_n^\imath$}
	For $\alpha \in \mathcal{V}_{p,n,d}^{\imath*}, \beta \in \mathcal{V}_{n,r,d}^{\imath*}$ and $\mj \in \mathbf{r}^d$, we have $$(\alpha\circ\beta) v_\mj=\sum_{\mi \in \mathbf{p}^d, \mk \in \Nid^d} \langle \alpha, \widetilde{t}_{\mi\mk} \rangle \langle \beta, \widetilde{t}_{\mk\mj} \rangle v_\mi,$$ which implies $$\langle \alpha\circ\beta, \widetilde{t}_{\mi\mj}\rangle =\sum_{\mk \in \Nid^d} \langle \alpha, \widetilde{t}_{\mi\mk} \rangle \langle \beta, \widetilde{t}_{\mk\mj} \rangle.$$ Therefore, we can write the dual map of $\Phi_{p,r,d}$ explicitly as follows:
	$$\Psi_{p,r,d}^\imath(:=\Phi_{p,r,d}^*): \mathcal{V}_{p,r,d}^\imath \rightarrow \mathcal{V}_{p,n,d}^\imath\otimes\mathcal{V}_{n,r,d}^\imath,\quad \widetilde{t}_{\mi\mj} \mapsto \sum_{\mk \in \Nid^d} \widetilde{t}_{\mi\mk} \otimes \widetilde{t}_{\mk\mj} \quad (\mi \in \underline{\mathbf{p}}^d, \mj \in \underline{\mathbf{r}}^d).$$
	Furthermore, we define a map $$\Psi_{p,r}^\imath=\prod_{d}\Psi_{p,r,d}^\imath: \mathcal{V}_{p,r}^\imath \rightarrow \mathcal{P}_{p,r}^\imath.$$

Since $\mathcal{P}_{p,r}^\imath$ is a $\mathcal{P}_{P,R}$-module, it is also a  $\mathcal{V}_{P,R}$-module canonically by the algebra homomorphism $\Psi_{P,R}$ appeared in Theorem~\ref{UN:FFTandSFT}. Moreover, The map $\Psi_{p,r}^\imath$ is a $\mathcal{V}_{P,R}$-module homomorphism.
	\begin{thm}\label{FFT}
		\begin{itemize}
			\item[(1)] The map $\Psi_{p,r}^\imath$ is a $(\qU_r^\imath,\qU_p^\imath)$-module homomorphism.
			\item[(2)] It holds that $\mathrm{Im}\Psi_{p,r}^\imath=\mathcal{X}_{p,r}^\imath$. Therefore, $\mathcal{X}_{p,r}^\imath$ is a quotient space of $\mathcal{V}_{p,r}^\imath$.
		\end{itemize}
	\end{thm}
	\begin{proof}
		Since $\Phi_{p,r,d}$ is a $(\qU_p^\imath, \qU_r^\imath)$-homomorphism by Lemma~\ref{im}, its dual $\Phi_{p,r,d}^{*}$ is canonically a $(\qU_r^\imath, \qU_p^\imath)$-homomorphism.
		
		Assume that $x_1, \ldots, x_k$ generate $\qU_n^\imath$. We define a map $$\varphi: \mathcal{V}_{p,n,d}^\imath \otimes \mathcal{V}_{n,r,d}^\imath \rightarrow (\mathcal{V}_{p,n,d}^\imath \otimes \mathcal{V}_{n,r,d}^\imath)^{\oplus k},\quad  f\otimes g \mapsto ((x_i f)\otimes g-f\otimes (g x_i))_{1 \leq i \leq k}.$$ By the definition of $\mathcal{X}_{p,r,d}^\imath$, we see $\mathcal{X}_{p,r,d}^\imath=\Ker\varphi$. Taking the dual, we obtain a map
		$$\varphi^*: (\mathcal{V}_{p,n,d}^{\imath*} \otimes \mathcal{V}_{n,r,d}^{\imath*})^{\oplus k} \rightarrow \mathcal{V}_{p,n,d}^{\imath*} \otimes \mathcal{V}_{n,r,d}^{\imath*} ,\quad  (\alpha_i \otimes \beta_i)_{1 \leq i \leq k} \mapsto \sum_{1 \leq i \leq k}\alpha_i x_i \otimes \beta_i-\alpha_i \otimes x_i \beta_i.$$
		Noting that $\mathrm{Coker}\varphi^*=\mathcal{V}_{p,n,d}^{\imath*} \otimes_{\qU_n^\imath} \mathcal{V}_{n,r,d}^{\imath*}\simeq\mathrm{Im}\Phi_{p,r,d}$ by Lemma \ref{im}, thus $$\mathrm{Im}\Psi_{p,r,d}^{\imath}=\mathrm{Im}\Phi_{p,r,d}^{*}\simeq\mathrm{Ker}\varphi=\mathcal{X}_{p,r,d}^\imath,$$
		which implies $\mathrm{Im}\Psi_{p,r}^\imath=\mathcal{X}_{p,r}^\imath$.
	\end{proof}
	
	\begin{rem}\label{remark1}
		Note that the space of $\qU_n^\imath$-invariants $\mathcal{X}_{p,r}^\imath$ is not an algebra.
		The above theorem implies that as a $\mathcal{V}_{P,R}$-module, $\mathcal{X}_{p,r}^\imath$ is generated by $\Psi_{p,r}^\imath(\widetilde{t}_{\mi\mj})=\sum_{\mk \in \Nid^d} \widetilde{t}_{\mi\mk} \otimes \widetilde{t}_{\mk\mj}$, $(\mi \in \underline{\mathbf{p}}^d, \mj \in \underline{\mathbf{r}}^d)$. It can be regarded as the FFT of $\qU_n^\imath$ on $\mathcal{P}_{p,r}^\imath$.
	\end{rem}
	
	\subsection{The SFT for $\qU_n^\imath$}
	Define two elements in $\HH_{B_{n^+}}$ and $\HH_{B_{n^-}}$ as follows:
	\begin{align*}
		u_n^+:=\sum_{w \in W_{B_{n^+}}} q^{-\ell_0(w)} (-q)^{\ell(w)-\ell_0(w)} T_w, \qquad
		u_n^-:=\sum_{w \in W_{B_{n^-}}} (-q)^{\ell(w)} T_w,
	\end{align*}
	where $\ell(w)$ is the length of $w$, and $\ell_0(w)$ is the number of $s_0$ appearing in a reduced form of $w$. The function $\ell_0$ is well-defined since it is a weight function (cf. \cite{Lu03}).
One can check that for any $w \in W_{B_{n^{\pm}}}$,
	\begin{align*}
		T_w u_n^+&=q^{-\ell_0(w)} (-q)^{\ell(w)-\ell_0(w)} u_n^+=u_n^+ T_w;\\
		T_w u_n^-&=(-q)^{\ell(w)} u_n^-=u_n^- T_w.
	\end{align*}
That is, the element $u_n^+$ (resp. $u_n^-$) is a central element in $\HH_{B_{n^+}}$ (resp. $\HH_{B_{n^-}}$).
	
	We set $\mathbb{W}_{m,n}^+:=\V_M^{\otimes n^+}u_n^+$ and $\mathbb{W}_{m,n}^-:=\V_M^{\otimes n^-}u_n^-$.
	
	Let
	\begin{align*}
		\mathfrak{A}_{n,m}^{\imath+}&:=\{\mi \in \Mid^{n^+} ~|~ 0 \leq i_1 < \cdots < i_{n^+} \leq m\};\\
		\mathfrak{A}_{n,m}^{\imath-}&:=\{\mi \in \Mid^{n^-} ~|~ 0 < i_1 < \cdots < i_{n^-} \leq m\}.
	\end{align*}
	
	\begin{lem}\label{basis}
		The set $\{v_\mi u_n^\pm ~|~ \mi \in \mathfrak{A}_{n,m}^{\imath\pm}\}$ forms a basis of $\mathbb{W}_{m,n}^\pm$.
	\end{lem}
	\begin{proof}
	We just prove the statement for $\mathbb{W}_{m,n}^+$.

For any $\mi \in \Mid^{n^+}$, there exists an $\mi'=(i_1,i_2,\ldots,i_{n^+}) \in \Mid^{n^+}$ such that $0 \leq i_1 \leq \cdots \leq i_{n^+} \leq m$ and $\mi=\mi'w$ for some $w \in W_{B_{n^+}}$ with the minimal length. We calculate that $v_\mi u_n^+=v_{\mi'}T_w u_n^+=q^{-\ell_0(w)} (-q)^{\ell(w)-\ell_0(w)} v_{\mi'} u_n^+$. Moreover, if $i_k'=i_{k+1}'$ for some $k$, then $(1+q^{-2})v_{\mi'} u_n^+ =v_{\mi'}(1+q^{-1}T_k)u_n^+=v_{\mi'}u_n^+-q^{-1}qv_{\mi'}u_n^+=0$. Thus, $\{v_\mi u_n^+ ~|~ \mi \in \mathfrak{A}_{n,m}^{\imath+}\}$ spans the space $\mathbb{W}_{m,n}^+$. Considering the limit $q \to 1$, for $\mi=(i_1,\ldots,i_{n^+})\in \mathfrak{A}_{n,m}^{\imath+}$, the element $v_\mi u_n^+$ specializes to $$\sum_{w\in \mathfrak{S}_{n^+}}(-1)^{\ell(w)}(v_{i_{w(1)}}+v_{-i_{w(1)}}) \otimes \cdots \otimes (v_{i_{w(n^+)}}+v_{-i_{w(n^+)}}) \in {(\fC^M)}^{\otimes n^+}.$$ So $\{v_\mi u_n^+ ~|~ \mi \in \mathfrak{A}_{n,m}^{\imath+}\}$ is linearly independent. Hence the lemma follows.
	\end{proof}
	
\begin{rem}\label{rem:W}
Let $V^+$ (resp. $V^-$) be the subspace of $\fC^M$ spanned by $v_i+v_{-i}$ (resp. $v_i-v_{-i}$), ($i\in\Mid$). The above proof shows that $\mathbb{W}_{m,n}^\pm$ are $q$-analogues of the wedge spaces $\bigwedge^{n^\pm} V^\pm$.
\end{rem}

	Recall in \S\ref{rew} that we give the definition of quantum minors $\Delta(\mi,\mj)$ for $\mi \in \mathfrak{A}_{N,P}$ and $\mj \in \mathfrak{A}_{N,R}$. Now for $\sigma \in \mathfrak{S}_N$, we define $\Delta(\mi,\mj \sigma):=(-q)^{\ell(\sigma)} \Delta(\mi,\mj)$.
	
	For $\mi \in \mathfrak{A}_{n,p}^{\imath+}$ and $\mj \in \mathfrak{A}_{n,r}^{\imath+}$, we define the \emph{positive $\imath$quantum $n$-minor} $\Delta^{\imath+}(\mi,\mj) \in \mathcal{V}_{p,r}^\imath$ by
	\begin{align*}
		\Delta^{\imath+}(\mi,\mj):=\sum_{|\mk|=\mj} q^{-\ell_0(\mk)} \widetilde{\Delta(\mi,\mk)},
	\end{align*}
	where $|\mk|:=(|k_1|,\ldots,|k_{n^+}|)$ and $\ell_0(\mk):=\sharp\{a ~|~ k_a<0\}$.
	For $\mi \in \mathfrak{A}_{n,p}^{\imath-}$ and $\mj \in \mathfrak{A}_{n,r}^{\imath-}$, we define the \emph{negative $\imath$quantum $n$-minor} $\Delta^{\imath-}(\mi,\mj) \in \mathcal{V}_{p,r}^\imath$ by
	\begin{align*}
		\Delta^{\imath-}(\mi,\mj):=\sum_{|\mk|=\mj} (-q)^{\ell_0(\mk)} \widetilde{\Delta(\mi,\mk)}.
	\end{align*}

	Let $\mathfrak{M}^{+}$ (resp. $\mathfrak{M}^{-}$) be the subspace of $\mathcal{V}_{p,r}^\imath$ spanned by all positive (resp. negative) $\imath$quantum $n$-minors.
		
	\begin{prop}
		The spaces $\mathfrak{M}^{\pm}$ are $(\qU_r^\imath, \qU_p^\imath)$-submodules of $\mathcal{V}_{p,r}^\imath$. Moreover, $\mathfrak{M}^{\pm} \cong \Hom(\mathbb{W}_{p,n}^\pm,\mathbb{W}_{r,n}^\pm)$
	\end{prop}
	\begin{proof}
		Define a $\K$-linear map
		$$\varrho_{p,r,d}: \mathbb{V}_R^{\otimes d} \rightarrow \mathbb{V}_P^{\otimes d} \otimes \mathcal{V}_{p,r,d}^\imath, \quad v_\mj \mapsto \sum_{\mi \in \underline{\mathbf{p}}^d} v_\mi \otimes \widetilde{t}_{\mi\mj}\quad \mbox{for $\mj \in \underline{\mathbf{r}}^d$}.$$
		Recall in \S\ref{subsecV} that $\mathcal{V}_{p,r,d}^{\imath*}$ is identified with $\Hom_{\HH_{B_d}}(\mathbb{V}_R^{\otimes d}, \mathbb{V}_P^{\otimes d})$ by $\rho_{p,r,d}$. For $\alpha \in \mathcal{V}_{p,r,d}^{\imath*}$ and $\mj \in \underline{\mathbf{r}}^d$, we have $\alpha(v_\mj)=\sum_{\mi \in \underline{\mathbf{p}}^d} \langle \widetilde{t}_{\mi\mj}, \alpha \rangle v_\mi$. Therefore,
		$$\varrho_{p,r,d}(v_\mj x)=\sum_{\mi \in \underline{\mathbf{p}}^d} (v_\mi x) \otimes \widetilde{t}_{\mi\mj},\quad \forall x \in \HH_{B_d}, \mj \in \underline{\mathbf{r}}^d.$$
		For each $\mj \in \mathfrak{A}_{n,r}^{\imath-}$ (In the case when $\mj \in \mathfrak{A}_{n,r}^{\imath+}$, we need to consider whether $j_0=0$, but the argument is similar. So for brevity, we just give a calculation for $\mj \in \mathfrak{A}_{n,r}^{\imath-}$ here), we have
		\begin{align*}
			&\varrho_{p,r,n^-}(v_\mj u_n^-)=\sum_{\mi \in \underline{\mathbf{p}}^{n^-}} (v_\mi u_n^-) \otimes \widetilde{t}_{\mi\mj}=\sum_{\mi \in \mathfrak{A}_{n,p}^{\imath-}} \sum_{w \in W_{B_{n^-}}} (v_\mi T_w u_n^-) \otimes \widetilde{t}_{\mi w,\mj}\\
			=&\sum_{\mi \in \mathfrak{A}_{n,p}^{\imath-}}  (v_\mi u_n^-) \otimes\sum_{w \in W_{B_{n^-}}} (-q)^{\ell(w)} \widetilde{t}_{\mi w,\mj}=\sum_{\mi \in \mathfrak{A}_{n,p}^{\imath-}} (v_\mi u_n^-) \otimes \Delta^{\imath-}(\mi ,\mj),
		\end{align*}
		which induces a surjective map (still denote by $\rho_{p,r,n^\pm}^*$) by Lemma \ref{basis}: $$\rho_{p,r,n^\pm}^*: \Hom(\mathbb{W}_{p,n}^\pm,\mathbb{W}_{r,n}^\pm) \rightarrow \mathfrak{M}^{\pm},\quad (v_\mi u_n^\pm)^* \otimes (v_\mj u_n^\pm) \mapsto \Delta^{\imath\pm}(\mi,\mj).$$ Here $\{(v_\mi u_n^\pm)^* ~|~ \mi \in \mathfrak{A}_{n,r}^{\imath\pm}\}$ is the dual basis of $\{(v_\mi u_n^\pm) ~|~ \mi \in \mathfrak{A}_{n,r}^{\imath\pm}\}$.
		
		Note that $\Hom(\mathbb{W}_{p,n}^\pm,\mathbb{W}_{r,n}^\pm)$ is a $(\qU_r^\imath, \qU_p^\imath)$-subquotient of $\Hom(\mathbb{V}_P^{\otimes n^\pm}, \mathbb{V}_R^{\otimes n^\pm})$. We have a commutative diagram as follows:
		$$\begin{tikzcd}
			{\Hom(\mathbb{W}_{p,n}^\pm,\mathbb{W}_{r,n}^\pm)} \arrow[d,two heads,"\rho_{p,r,n^\pm}^*"]& {\Hom(\mathbb{V}_P^{\otimes n^\pm}, \mathbb{W}_{r,n}^\pm)} \arrow[r, hook] \arrow[l,two heads] & {\Hom(\mathbb{V}_P^{\otimes n^\pm}, \mathbb{V}_R^{\otimes n^\pm})} \arrow[d,"\rho_{p,r,n^\pm}^*"]&\\
			\mathfrak{M}^{\pm} \arrow[rr,hook] && \mathcal{V}_{p,r,n^\pm}^\imath.&\\
		\end{tikzcd}$$
		Therefore, $\mathfrak{M}^{\pm}$ are $(\qU_r^\imath, \qU_p^\imath)$-submodules of $\mathcal{V}_{p,r}^\imath$.
Comparing the dimensions, we know $\dim_\K(\mathfrak{M}^{\pm})=\#\mathfrak{A}_{n,p}^{\imath\pm}\#\mathfrak{A}_{n,r}^{\imath\pm}=\dim_\K(\mathbb{W}_{p,n}^\pm)\dim_\K(\mathbb{W}_{r,n}^\pm)$. So it must be that $\mathfrak{M}^{\pm} \cong \Hom(\mathbb{W}_{p,n}^\pm,\mathbb{W}_{r,n}^\pm)$.
	\end{proof}
	
\begin{rem}
The above proof shows that elements in $\mathfrak{M}^{\pm}$ can be identified by linear functions in $\Hom(\mathbb{W}_{p,n}^\pm,\mathbb{W}_{r,n}^\pm)$. Recall in Remark~\ref{rem:W} that $\mathbb{W}_{m,n}^\pm$ are $q$-analogues of the wedge spaces $\bigwedge^{n^\pm} V_m^\pm$. Thus at the classical limit $q\rightarrow1$, the spaces $\mathfrak{M}^\pm$ specialize to $\Hom(\bigwedge^{n^\pm} V_p^\pm,\bigwedge^{n^\pm} V_r^\pm)$, which are canonically isomorphic to the subspaces spanned by the $n^\pm$-minors of $\mathrm{Mat}_{p^\pm \times r^\pm}$.
\end{rem}

	Let $\mathcal{I}$ be the $\mathcal{V}_{P,R}$-submodule of $\mathcal{V}_{p,r}^\imath$ generated by $\mathfrak{M}^+$ and $\mathfrak{M}^-$.
	\begin{thm}\label{SFT}
		\begin{itemize}
			\item[(1)] The space $\mathcal{I}$ is a $(\qU_r^\imath, \qU_p^\imath)$-submodule of $\mathcal{V}_{p,r}^\imath$.
			\item[(2)] The space $\mathcal{I}$ admits the multiplicity-free decomposition as a $(\qU_r^\imath, \qU_p^\imath)$-module:
			\begin{align*}
				\mathcal{I} \simeq \bigoplus_{\lambda \in \Par_{\min\{p,r\}}^\imath \backslash \Par_n^\imath} L_\lambda^{r,\imath} \otimes L_\lambda^{p,\imath*},
			\end{align*}
			which implies $\mathcal{V}_{p,r}^\imath/\mathcal{I}$ admits the following multiplicity-free decomposition as a $(\qU_r^\imath, \qU_p^\imath)$-module:
			\begin{align*}
				\mathcal{V}_{p,r}^\imath/\mathcal{I} \simeq \bigoplus_{\lambda \in \Par_{\min\{n,p,r\}}^\imath} L_\lambda^{p,\imath} \otimes L_\lambda^{r,\imath*}.
			\end{align*}
			\item[(3)] $\Ker \Psi_{p,r}^\imath=\mathcal{I}$.
		\end{itemize}
	\end{thm}
	\begin{proof}
		By \cite[Lemma~5.1]{LX22}, we see $x(fg)y=(x_{(1)}fy_{(1)})(x_{(2)}gy_{(2)})$ for $f \in \mathfrak{M}^{\pm}$, $g \in \mathcal{V}_{P,R}$, $x \in \qU_r^\imath$, $y \in \qU_p^\imath$. Since the $\imath$quantum groups are right coideal subalgebras, i.e., $x_{(1)} \in \qU_r^\imath$ and $y_{(1)} \in \qU_s^\imath$, we have $x_{(1)}fy_{(1)} \in \mathfrak{M}^{\pm}$ and hence $x(fg)y \in \mathcal{I}$. Thus $\mathcal{I}$ is a $(\qU_r^\imath, \qU_p^\imath)$-module. This proves the statement (1).
		
	At the classical limit $q\rightarrow1$, the space $\mathcal{I}$ specializes to the ideal of $\mathrm{sym}(\fC^{p^+} \otimes {\fC^{r^+}}^*) \oplus \mathrm{sym}(\fC^{p^-} \otimes {\fC^{r^-}}^*)$ generated by $n^\pm$-minors in $\mathrm{Mat}_{p^\pm \times r^\pm}$. By considering the $(\mathfrak{gl}_{r^\pm},\mathfrak{gl}_{{p^\pm}})$-module decomposition of the ideal of $\mathrm{sym}(\fC^{p^\pm} \otimes {\fC^{r^\pm}}^*)$ generated by $n^\pm$-minors, we know that the specialization of $\mathcal{I}$ admits the decomposition
		$$\bigoplus_{\tiny \begin{array}{l}\lambda \in \Par_{\min\{p^+,r^+\}} \backslash \Par_{n^+}, \\ \mu \in \Par_{\min\{p^-,r^-\}} \backslash \Par_{n^-}\end{array}} L_\lambda^{r^+} \otimes L_\lambda^{p^+*} \oplus L_\mu^{r^-} \otimes L_\mu^{p^-*},$$ from which the desired two decompositions in the statement (2) follow. Then the statement~(3) is obtained by Lemma~\ref{fd} and Theorem~\ref{FFT}.
	\end{proof}
	\begin{rem}
		As we mentioned in Remark~\ref{remark1}, the space of $\qU_n^\imath$-invariants $\mathcal{X}_{p,r}^\imath$ is not an algebra. So it is not applicable to describe the SFT as generating relations of generators. Instead, we provide $\Ker \Psi_{p,r}^\imath$ which is an imitation of the description of SFT for $\qU_N$ in \cite{Z20}.
	\end{rem}

\end{document}